\documentclass[10pt]{article}
\usepackage[utf8]{inputenc} 
\usepackage[T1]{fontenc}    
\usepackage[hidelinks]{hyperref}       
\usepackage{url}            
\usepackage{booktabs}       
\usepackage{amsfonts}       
\usepackage{nicefrac}       
\usepackage{microtype}      
\usepackage[table]{xcolor}
\usepackage{natbib}
\usepackage{fullpage}
\usepackage{multirow}
\usepackage{soul}           

\newcommand{\calA}{\mathcal{A}}

\newcommand{\calS}{\mathcal{S}}

\newcommand{\E}{\mathbb{E}}

\newcommand{\ls}{\left}
\newcommand{\rs}{\right}

\newcommand{\interior}{\mathrm{int}}

\newcommand{\dom}{\mathrm{dom}}
\newcommand{\inner}[2]{\ls\langle #1, \; #2 \rs\rangle}

\newcommand\numberthis{\addtocounter{equation}{1}\tag{\theequation}}
\makeatletter

\newcommand{\Rmnum}[1]{\textup{\expandafter\@slowromancap\romannumeral #1@}}

\makeatother

\usepackage[table]{xcolor}
\usepackage{tabstackengine}
\usepackage{appendix}
\usepackage{titletoc}
\usepackage{titlesec}
\usepackage{enumerate}
\usepackage{amsmath}
\usepackage{amssymb}
\usepackage{mathtools}
\usepackage{amsthm}
\usepackage{algorithm}
\usepackage{algorithmic}
\usepackage{graphicx}
\usepackage{makecell}
\usepackage{array}
\usepackage{threeparttable}
\usepackage{setspace}
\usepackage{dsfont}
\usepackage{multicol}
\usepackage{authblk}
\usepackage{lipsum}
\usepackage{tabularx}
\usepackage{enumitem}
\usepackage{subcaption}
\allowdisplaybreaks

\theoremstyle{plain}
\newtheorem{theorem}{Theorem}[section]
\newtheorem{proposition}[theorem]{Proposition}
\newtheorem{lemma}[theorem]{Lemma}

\theoremstyle{definition}

\newtheorem{assumption}[theorem]{Assumption}

\title{Does Policy Mirror Descent Converge in Policy with Non-Decomposable Mirror Maps? A Counterexample}
\author[1]{Wenye Li}
 \author[1]{Ke Wei}
\affil[1]{School of Data Science, Fudan University, Shanghai, China.}
\date{\today}

\begin{document}
\maketitle

\begin{abstract}
    Policy mirror descent (PMD) refers to a class of policy optimization methods associated with different mirror maps. For a broad class of decomposable Legendre mirror maps encompassing many standard choices, \citet{Li2026PMDPolicy} show that the PMD policy sequence converges to an optimal policy from any admissible initialization under any positive constant step size, even when the optimal policy is not unique.  In this note, we provide a counterexample with assistance from OpenAI's GPT-5.6 model, showing that PMD with a non-decomposable Legendre mirror map need not converge in policy. 
\end{abstract}



\section{Introduction}
Let $\Delta(\mathcal{X})$ denote the probabilistic simplex over some finite set $\mathcal{X}$. A finite discounted Markov decision process (MDP) is specified by a tuple $(\calS,\calA,P,r,\gamma,\mu)$, where $\calS$ and $\calA$ are finite state and action spaces, respectively, $P:\calS\times\calA\to\Delta(\calS)$ is the transition kernel, $r:\calS\times\calA\to[0,1]$ is the reward function, $\gamma\in[0,1)$ is the discount factor, and $\mu\in\Delta(\calS)$ is the initial state distribution. Given a policy $\pi:\calS\to\Delta(\calA)$, a trajectory $\{(s_t,a_t,r_t)\}_{t\ge0}$ is generated as follows: the initial state is sampled from $s_0\sim\mu$, and at each time step $t\ge0$, the agent selects an action $a_t\sim\pi(\cdot\mid s_t)$, receives the reward $r_t=r(s_t,a_t)$, and transitions to the next state $s_{t+1}\sim P(\cdot\mid s_t,a_t)$. The state value and the action value functions are defined as 
\begin{align*}
    \forall\, s \in \calS,\; a\in\calA: \quad 
    &V^{\pi}(s) := \E \, \Bigg[ \sum_{t=0}^\infty \gamma^{t} r(s_t, a_t) \; \Big | \; s_0 = s, \, \pi  \Bigg], \\
    &Q^{\pi}(s,a) := \E \, \Bigg[ \sum_{t=0}^\infty \gamma^{t} r(s_t, a_t) \; \Big | \; (s_0, a_0) = (s, a), \, \pi  \Bigg].
\end{align*}

Policy mirror descent (PMD) updates policies using a Bregman divergence induced by a convex mirror map $h:\mathbb{R}^{|\calA|}\to\mathbb{R}\cup\{+\infty\}$. Assume that $\Delta(\calA)\subseteq\operatorname{dom}h$ and that $h$ is differentiable on $\operatorname{int}(\operatorname{dom}h)$. Given a current policy $\pi$ satisfying $\pi(\cdot\mid s)\in\operatorname{int}(\operatorname{dom}h)$ for all $s\in\calS$, PMD computes a new policy based on the action value function $Q^\pi$ by solving the following optimization problem at each state:
\begin{align}  \label{eq:PMD}
    \mbox{(PMD)} \quad \forall\, s\in\calS: \;\;  \pi^+(\cdot|s) \in \underset{p\in\Delta(\calA)}{\arg\max} \; \Big \{ \eta \,  \big \langle p, \; Q^{\pi}(s,\cdot) \big \rangle - D_h(p \, \| \, \pi(\cdot|s)) \Big\},
\end{align}
where $\eta > 0$ is the update step size and $D_h(p \, \| \, q) := h(p) - h(q) - \inner{\nabla h(q)}{p-q})$ is the Bregman divergence. Different choices of the mirror map $h$ yield different policy optimization methods. In particular, when $h$ is the squared function $h(p) = \frac{1}{2} \| p \|_2^2$ or the negative Shannon entropy $h(p) = \sum_{a\in\calA} p(a) \log p(a)$,  PMD deduces to projected Q-ascent (PQA,~\citet{Xiao_2022,ppgliu}) and softmax natural policy gradient (softmax NPG,~\citet{Xiao_2022,Agarwal_Kakade_Lee_Mahajan_2019}), respectively. The global convergence in value of PMD has been extensively studied. When $h$ is of Legendre type, it is shown in~\citet{Xiao_2022,Lan_2021} that the state value function of PMD $V^{\pi_k}$ converges to the optimal value $V^*$ in a dimension-free $O(1/k)$ sublinear rate under arbitrary constant step size $\eta > 0$. For  adaptive or geometrically increasing step sizes, the global linear convergence can be established~\citep{Xiao_2022,Johnson_Pike-Burke_Rebeschini_2023,yuan2023general}. Several variants of PMD have also been investigated, including regularized PMD~\citep{Lan_2021,Cen_Cheng_Chen_Wei_Chi_2022,Zhan_Cen_Huang_Chen_Lee_Chi_2021}, homotopic PMD~\citep{Li_Zhao_Lan_2022}, TD-PMD~\citep{liu2025tdpmd,Li2026Markov-TDPMD,Geist_Scherrer_Pietquin_2019}.

Beyond value convergence, policy convergence of PMD has also been studied. For decomposable mirror maps of the form $h(p)=\sum_{a\in\calA}\psi(p(a))$, where $\psi$ satisfies standard Legendre-type assumptions, \citet{Li2026PMDPolicy} establish that the policy sequence generated by PMD converges to a limiting optimal policy under any constant step size $\eta>0$, even when the optimal policy is not unique. For mirror maps that need not be decomposable, \citet{Lin2022PMD-policy-convergence} establish finite-time convergence to an optimal policy under any constant step size, provided that $h$ is continuously differentiable on the entire probability simplex $\Delta(\calA)$ and satisfies an $L$-cocoercivity condition. These results raise the following question:
\begin{quote}\centering\it
    Does PMD converge in policy for every non-decomposable
    Legendre mirror map?
\end{quote}

\textbf{This note, however, gives a negative answer by providing a counter example.} The counterexample is based on a non-decomposable mirror map $h$ that combines algebraic, logarithmic, and trigonometric terms and satisfies the Legendre conditions: it is proper, closed, convex, essentially strictly convex, and essentially smooth. The associated MDP is a three-armed bandit with $\calS=\{s\}$ and $\calA=\{a_1,a_2,a_3\}$, where $a_1$ and $a_2$ are optimal actions and $a_3$ is suboptimal.  By computation, it is shown that the policy sequence generated by PMD  from a particular initialization on this example oscillates infinitely times due to the periodicity of trigonometric component in $h$.

\section{Preliminary and Problem Settings}
Let $\Pi$ be the set that contains all admissible policies. Recall the value functions $V^\pi$ and $Q^\pi$. It is well known that there exists at least one optimal policy $\pi^*$, such that the optimal value functions $V^*$ and $Q^*$ are obtained (see for example~\citet{suttonRL}),
\begin{align*}
    \forall\, (s,a) \in \calS\times\calA: \quad V^{\pi^*}(s) = V^*(s) := \max_{\pi} \, V^{\pi}(s), \;\;\; Q^{\pi^*}(s,a) = Q^*(s,a) := \max_\pi \, Q^\pi(s,a).
\end{align*}
Note that the optimal policy $\pi^*$ might not be unique. In fact, it is not hard to see that the set of all optimal policies can be represented as
\begin{align*}
    \big\{ \pi^* \; | \;  V^{\pi^*} = V^*, \; Q^{\pi^*} = Q^* \big\} &= \big\{ \pi^* \in \Pi \; | \; \pi^*(a^\prime|s) = 0, \; \forall\, s\in\calS, \; \forall\, a^\prime\not\in\calA^*_s \big\} \\
    &= \big\{ \pi^* \in \Pi \; | \; \mbox{supp}(\pi^*(\cdot|s)) \subseteq \calA^*_s \big\},
\end{align*}
where $\calA^*_s := \{ a\in\calA\, | \, Q^*(s,a) = V^*(s) \}$ is the optimal action set. Thus $\Pi^*$ contains multiple optimal policies if $|\calA^*_s| > 1$ at some state $s$.

Recall that $h$ denotes the mirror map used in PMD. Generally, we expect that $h$ is of Legendre type~\citep{bauschke1997legendre,Xiao_2022} and is well defined over $\Delta(\calA)$.

\begin{assumption}[Legendreness]  \label{ass:Legendreness}
    We assume that $h$ satisfies the following conditions:
    \begin{enumerate}[label=(\alph*)]
        \item \label{ass:Legendreness-domain} $\dom \, h \supseteq \Delta(\calA)$;
        \item \label{ass:Legendreness-proper} $h$ is proper, closed, and convex;
        \item \label{ass:Legendreness-smooth} $h$ is essentially smooth. That is, $h$ is differentiable over $\interior\, \dom \, h$, and for any sequence $\{ x_k \}_{k \geq 0}$ such that $x_k \in \interior\, \dom\, h$ and $x_k \to x \in \mbox{bd} \, \dom \, h$, there holds $\| \nabla h(x_k) \| \to \infty$;
        \item \label{ass:Legendreness-strictly-convex} $h$ is strictly convex over $\interior\, \dom \, h$.
    \end{enumerate}
\end{assumption}

The following lemma is useful in the PMD analysis, which can be verified by the KKT-condition of the PMD optimization problem (Eq~\eqref{eq:PMD}).
\begin{lemma}[\citep{Lin2022PMD-policy-convergence}]  \label{lem:KKT}
    Let $\{ \pi_k \}_{k \geq 0}$ be the policy sequence generated by PMD \textup(Eq~\eqref{eq:PMD}\textup) with $\pi_0(\cdot|s) \in \big( \interior\, \dom \, h \big) \cap \Delta(\calA)$ for all $s\in\calS$. Then there hold $\pi_k(\cdot|s) \in  \big( \interior\, \dom \, h \big) \cap \Delta(\calA)$ for all $s\in\calS$ and $k \geq 0$. Furthermore, there exist a scalar $\nu^{(k)}$ and a dual vector $\lambda^{(k)} \in \mathbb{R}^{|\calA|}$ such that
    \begin{align*}
        \forall\, a\in \calA: \quad \frac{\partial h}{\partial p(a)} (\pi_{k+1}(\cdot|s)) = \frac{\partial h}{\partial p(a)} (\pi_k(\cdot|s)) + \eta \, Q^{\pi_k}(s,a) + \lambda^{(k)}(a) + \nu^{(k)},
    \end{align*}
    where $\lambda^{(k)}(a)$ satisfies $\lambda^{(k)}(a) \geq 0$ and $\lambda^{(k)}(a) \cdot \pi_{k+1}(a|s) = 0$.
\end{lemma}

\begin{proposition}[\citet{Xiao_2022,Lan_2021}]  \label{pro:Xiao_Lan}
    Suppose that $h$ satisfies Assumption~\ref{ass:Legendreness}. Let $\{ \pi_k \}_{k \geq 0}$ and $\{ V^{\pi_k} \}_{k \geq 0}$ be the policy and value sequence generated by PMD \textup(Eq~\eqref{eq:PMD}\textup) with $\pi_0(\cdot|s) \in \big( \interior\, \dom \, h \big) \cap \Delta(\calA)$ for all $s\in\calS$, respectively. Then there hold $\pi_k(\cdot|s) \in  \big( \interior\, \dom \, h \big) \cap \Delta(\calA)$ for all $s\in\calS$ and $k \geq 0$. Furthermore,
    \begin{alignat*}{2}  
        &\forall\, s\in\calS: \quad &&V^{\pi_{k+1}}(s) \geq V^{\pi_k}(s), \numberthis \label{eq:PMD-monotonicity-V} \\
        &\forall\, s\in\calS, \; \forall \, a\in\calA: \quad && Q^{\pi_{k+1}}(s,a) \geq Q^{\pi_k}(s,a). \numberthis \label{eq:PMD-monotonicity-Q}
    \end{alignat*}
    Furthermore, for arbitrary optimal policy $\pi^* \in \Pi^*$, it holds that 
    \begin{align*}
        \big\| V^* - V^{\pi_k} \big\|_\infty \leq \frac{1}{k+1} \Bigg( \frac{\| D^{\pi^*}_{\pi_0} \|_\infty}{\eta(1-\gamma)} + \frac{1}{(1-\gamma)^2} \Bigg),
    \end{align*}
    where $\| D^{\pi^*}_{\pi_0} \|_\infty = \max_{s\in\calS} \, D_h (\pi^*(\cdot|s) \, \| \, \pi_0(\cdot|s)) < \infty$.
\end{proposition}
Proposition~\ref{pro:Xiao_Lan} establishes the dimension-free $O(1/k)$ sublinear global value convergence rate for PMD under Legendre mirror map $h$. As optimal policy might not be unique, Proposition~\ref{pro:Xiao_Lan} does not imply that $\{ \pi_k \}_{k \geq 0}$ converges to some optimal policy $\pi^* \in \Pi^*$. To address this question, \citet{Li2026PMDPolicy} establish policy convergence of PMD under any constant step size $\eta>0$ for a general class of decomposable mirror maps $h(p)=\sum_{a\in\calA}\psi(p(a))$, where $\psi$ satisfies standard Legendre-type assumptions. Their result guarantees  cover many commonly used mirror maps, including the squared Euclidean mirror map ($\psi(t) = \frac{1}{2} t^2$), negative Shannon entropy ($\psi(t) = t \log t$), Tsallis entropy ($\psi(t)=t^q$), the Hellinger mapping ($\psi(t)=-\sqrt{1-t^2}$), and Fermi--Dirac entropy ($\psi(t) = t\log t + (1-t) \log (1-t)$).
For mirror maps that need not be decomposable, \citet{Lin2022PMD-policy-convergence} establish finite-time convergence to an optimal policy under any constant step size, provided that $h$ is continuously differentiable on the entire simplex $\Delta(\calA)$ and satisfies an $L$-cocoercivity condition. 

\begin{proposition}[\citep{Lin2022PMD-policy-convergence}]
    Suppose that $h$ satisfies Assumption~\ref{ass:Legendreness}. Furthermore, suppose that $\nabla h(p)$ is well defined at $p \in \mbox{bd} \, \Delta(\calA)$ and $L$-coercive on $\Delta(\calA)$, i.e., there exists a constant $L> 0$ such that
    \[
        \forall\, p, \,q \in\Delta(\calA): \quad \big\langle \nabla h(p) - \nabla h(q), \; p - q \big\rangle \geq \frac{1}{L} \big\| \nabla h(p) - \nabla h(q) \big\|_2^2.
    \]
    Let $\{ \pi_k \}_{k \geq 0}$ be the policy sequence generated by PMD \textup(Eq~\eqref{eq:PMD}\textup) with any $\eta > 0$ and $\pi_0(\cdot|s) \in \big( \interior\, \dom \, h \big) \cap \Delta(\calA)$ for all $s\in\calS$. Then there exists a finite time $T_0$ such that
    \[
        \forall\, k \geq T_0: \quad \pi_k = \pi_{T_0} \in \Pi^*,
    \]
    i.e., PMD terminates after $T_0$.
\end{proposition}

\begin{proposition}[\citep{Li2026PMDPolicy}]
    Suppose that $h$ satisfies Assumption~\ref{ass:Legendreness} and is decomposable, i.e., $h$ can be written as $h(p) = \sum_{a\in\calA} \psi(p_a)$. Let $\{ \pi_k \}_{k \geq 0}$ be the policy sequence generated by PMD \textup(Eq~\eqref{eq:PMD}\textup) with any $\eta > 0$ and $\pi_0(\cdot|s) \in \big( \interior\, \dom \, h \big) \cap \Delta(\calA)$ for all $s\in\calS$. Then
    \begin{align}   \label{eq:policy-convergence}
        \pi_\infty := \lim_{k\to\infty} \, \pi_k \;\; \mbox{exists,} \quad \mbox{and } \; \pi_\infty \in \Pi^*.
    \end{align}
    That is, PMD converges in the policy domain.
\end{proposition}

It is worth noting that, although the result in \citet{Lin2022PMD-policy-convergence} does not require decomposability, the regularity assumptions therein exclude several standard decomposable mirror maps that are covered in \citet{Li2026PMDPolicy}. In particular, negative Shannon entropy, the Hellinger mapping, and Fermi--Dirac entropy are not differentiable at certain boundary points of the simplex. Moreover, Tsallis entropy with $1<q<2$ is continuously differentiable but fails the cocoercivity condition for every finite $L$.


\section{Policy Convergence with Non-Decomposable Legendre Mirror Maps? A Counterexample}

\subsection{Proposed bandit problem}
The proposed three-armed bandit problem is described as follows:
\begin{equation}
\begin{gathered}  \label{eq:bandit}
    \calS = \{ s \}, \quad \calA = \{ a_1,\,  a_2, \, a_3 \}, \quad \gamma=0, \quad P(s \, | \, s, a_i) = 1, \; (i=1,2,3), \\
    r(s,a_1) = r(s,a_2) = 1, \quad r(s,a_3) = 0.
\end{gathered}
\end{equation}
For any policy $\pi$, 
\begin{align*}
    V^\pi(s) = \pi(a_1|s) + \pi(a_2|s), \quad Q^{\pi}(s,a_1) = Q^\pi(s,a_2) = 1, \quad Q^\pi(s,a_3) = 0.
\end{align*}
A policy is optimal if and only if $\pi(a_3|s) = 0$, and thus the optimal policy for  this bandit problem is not unique.

\subsection{Proposed mirror map}
For any $p \in \Delta(\calA)$ where $\calA = \{ a_1, a_2, a_3 \}$ is the action space of the bandit problem above, denote $p_i = p(a_i)$. The proposed mirror map instance $h: \, \mathbb{R}^3 \to \mathbb{R}$ is given by
\begin{equation}  \label{eq:counter-h}
\begin{aligned}
    h(p) &= p_3 \log p_3 + 100 p_3 ^ 2 - \sqrt{4-p_3^2} \\
    &\;\;\;\;- p_3^4 \sqrt{1-(p_1-p_2)^2} - p_3^4(p_1 - p_2) \cdot \frac{1}{4} \sin (\log (1-\log p_3)) + \frac{1}{2} (p_1 + p_2 + p_3)^2,
\end{aligned}
\end{equation}
with the domain (with $0\log 0=0$ and $p_3^4\sin(\log(1-\log p_3))$ defined to be zero at $p_3=0$)
\begin{align*}
    \dom\, h = \big\{ p \in \mathbb{R}^3 \; | \; 0 \leq p_3 \leq 2, \;\; |p_1 - p_2| \leq 1 \big\}.
\end{align*}
Clearly, the mirror map $h$ is non-decomposable and does not satisfy the regularity conditions in~\citet{Lin2022PMD-policy-convergence}, since $\nabla h(p)$ is not defined on the boundary face $p_3=0$. We next verify that $h$ satisfies Assumption~\ref{ass:Legendreness}.

\paragraph{Domain and continuity} It is direct to verify that $\dom\, h \supseteq \Delta(\calA)$, satisfying Assumption~\ref{ass:Legendreness}~\ref{ass:Legendreness-domain}. Furthermore, $h$ is proper and closed as it is continuous over its closed domain.

\paragraph{Gradient and essentially smoothness} Obviously, $h$ is differentiable over $\interior\, \dom \, h = \big\{ p\in\mathbb{R}^3 \, | \, 0 < p_3 < 2, \;\; |p_1-p_2| < 1 \big\}$. For any $p \in \interior\, \dom \, h$, the gradient $\nabla h(p) = \big[ \frac{\partial h}{\partial{p_1}}(p), \, \frac{\partial h}{\partial{p_2}}(p), \, \frac{\partial h}{\partial{p_3}}(p) \big]$ is given by
\begin{equation}
\begin{aligned}  \label{eq:gradients}
    \dfrac{\partial h}{\partial p_1}(p) &= H(p) + p_1 + p_2 + p_3, \\
    \dfrac{\partial h}{\partial p_2}(p) &= -H(p) + p_1 + p_2 + p_3, \\
    \dfrac{\partial h}{\partial p_3}(p) &= 1 + \log p_3 + 200 p_3 + \dfrac{p_3}{\sqrt{4-p_3^2}} - 4p_3^3\sqrt{1-(p_1 - p_2)^2} \\
    &  \;\;\;\; - (p_1 - p_2) p_3^3 \Bigg[ \sin \big( \log (1- \log p_3) \big) - \dfrac{\cos \big( \log( 1 - \log p_3) \big)}{4(1- \log p_3)} \Bigg] \\
    &  \;\;\;\; + p_1 + p_2 + p_3,
\end{aligned}
\end{equation}
where $H(p)$ is the potential function
\begin{align}  \label{eq:potential-function}
    H(p) := \dfrac{p_3^4(p_1 - p_2)}{\sqrt{1-(p_1 - p_2)^2}} - \frac{p_3^4}{4} \sin \big( \log (1 - \log p_3) \big).
\end{align}
Notice that the boundary of $\dom\, h$ is given by $\mbox{bd} \, \dom\, h = \big\{ p \in \dom\, h \, | \, p_3=0 \;\; \mbox{or} \;\; p_3 = 2 \;\; \mbox{or} \;\; |p_1-p_2| = 1 \big\}$. For every sequence $\{p_k\}\subseteq\operatorname{int}(\dom \,h)$ converging to a point in $\operatorname{bd}(\dom \,h)$, one has $\|\nabla h(p_k)\|_2\to+\infty$. Indeed, $\partial_{p_3}h(p_k)\to-\infty$ when $(p_k)_3\to0$, and $\partial_{p_3}h(p_k)\to+\infty$ when $(p_k)_3\to2$. At any remaining boundary point, $|(p_k)_1-(p_k)_2|\to1$ while $(p_k)_3$ converges to a point in $(0,2)$, so $|\partial_{p_1}h(p_k)-\partial_{p_2}h(p_k)|=2|H(p_k)|\to+\infty$.  Thus $h$ is essentially smooth, which satisfies Assumption~\ref{ass:Legendreness}~\ref{ass:Legendreness-smooth}. 

\paragraph{Convexity} By the Hessian matrix of $h$, one can verify that (see the appendix for details)
\begin{align*}
    \nabla^2 h(p) \succ 0, \quad \forall\, p \in \interior\, \dom\, h = \big\{ p \in \mathbb{R}^3 \, | \, 0 < p_3 < 2, \;\; |p_1 - p_2| < 1 \big\}.
\end{align*}
Thus, $h$ is strictly convex over $\interior\, \dom \, h$ and convex over $\dom\, h$. Together with the fact that $h$ is proper and closed, we conclude that $h$ satisfies Assumptions~\ref{ass:Legendreness}~\ref{ass:Legendreness-proper} and~\ref{ass:Legendreness}~\ref{ass:Legendreness-strictly-convex}. 

\subsection{Nonconvergence of the PMD Policy Sequence}
In the sequel, let $\{ \pi_k \}_{k \geq 0}$ be the policy sequence generated by PMD under aforementioned mirror map and bandit problem, with any given $\eta > 0$ and the initialization 
\begin{align}  \label{eq:init}
    \pi_0(a_3|s) = 0.2, \quad \pi_0(a_1 | s) = \frac{0.8 + x_0}{2}, \quad \pi_0(a_2|s) = \frac{0.8 - x_0}{2}, \quad x_0 = \frac{\sin(\log ( 1 - \log 0.2))}{\sqrt{16 + \big( \sin(\log ( 1 - \log 0.2)) \big)^2}}.
\end{align}
Additionally, as $\pi_k(\cdot|s) \in \interior\, \dom \, h$ (see Proposition~\ref{pro:Xiao_Lan}), we have $\pi_k(a_3|s) > 0$ for all $k \geq 0$.
\begin{lemma}  \label{lem:potential}
    Recall the potential function $H(p)$. Under the initialization in Eq~\eqref{eq:init}, we have
    \begin{align}  \label{eq:potential}
        \forall\, k \geq 0: \quad H(\pi_k(\cdot|s)) = 0.
    \end{align}
   Moreover, there holds
    \begin{align}  \label{eq:policy-expression}
        \forall\, k \geq 0: \quad \pi_k(a_1 |s) = \frac{(1-\pi_k(a_3|s)) + u_k}{2}, \quad \pi_k(a_2|s) = \frac{(1 - \pi_k(a_3|s)) - u_k}{2},
    \end{align}
    where 
    \begin{align*}
        u_k = \frac{\sin(\log(1-\log \pi_k(a_3|s)))}{\sqrt{16+\sin^2(\log(1-\log \pi_k(a_3|s)))}} \in \ls[ -\frac{1}{\sqrt{17}}, \; \frac{1}{\sqrt{17}} \rs].
    \end{align*}
\end{lemma}

\begin{proof}
    We first prove Eq~\eqref{eq:potential} by induction. For the case of $k=0$, plugging the initialization (Eq~\eqref{eq:init}) into the definition of the potential function $H(p)$ (Eq~\eqref{eq:counter-h}) provides that $H(\pi_0(\cdot|s)) = 0$. Now suppose that $H(\pi_k(\cdot|s)) = 0$ for some $k \geq 0$. By Lemma~\ref{lem:KKT}, there exist a scalar $\nu^{(k)}$ and a dual vector $\lambda^{(k)} \in \mathbb{R}^{3}$ such that
    \begin{align}  \label{eq:KKT}
        \forall\, k \geq 0, \; \forall\, i = 1, 2, 3: \quad \frac{\partial h}{\partial p_i} (\pi_{k+1}(\cdot|s)) = \frac{\partial h}{\partial p_i} (\pi_k(\cdot|s)) + \eta \, Q^{\pi_k}(s,a_i) + \lambda_i^{(k)} + \nu^{(k)},
    \end{align}
    where $\lambda_i^{(k)}$ satisfies $\lambda_i^{(k)} \geq 0$ and $\lambda_i^{(k)} \cdot \pi_{k+1}(a_i|s) = 0$. For $i=1,2$, plugging Eq~\eqref{eq:gradients} into Eq~\eqref{eq:KKT} yields
    \begin{equation}  \label{eq:potential-equations}
    \begin{aligned}
        \begin{cases}
            H(\pi_{k+1}(\cdot|s)) = H(\pi_k(\cdot|s)) + \eta + \lambda_1^{(k)} + \nu^{(k)}, \\[.5em]
            H(\pi_{k+1}(\cdot|s)) = H(\pi_k(\cdot|s)) - \eta - \lambda_2^{(k)} - \nu^{(k)}.
        \end{cases}
    \end{aligned}
    \end{equation}
    Thus we obtain that
    \begin{align}  \label{eq:potential-dynamic}
        \forall\, k \geq 0: \quad H(\pi_{k+1}(\cdot|s)) = H(\pi_k(\cdot|s)) + \frac{\lambda_1^{(k)} - \lambda_2^{(k)}}{2}.
    \end{align}
    Then it remains to show that $\lambda_1^{(k)} = \lambda_2 ^{(k)} = 0$. Recall that $\pi_0(a_3|s) = 0.2$. By the monotonicity property of PMD (Eq~\eqref{eq:PMD-monotonicity-V} in Proposition~\ref{pro:Xiao_Lan}), there holds $\pi_k(a_3|s) \leq 0.2$, and thus $\pi_k(a_1|s) + \pi_k(a_2|s) \geq 0.8$. Now suppose that $\lambda_1^{(k)}>0$. In this case, we have $\pi_{k+1}(a_1|s) = 0$, thus $\pi_{k+1}(a_2|s) \geq 0.8$ and $\lambda_2^{(k)} = 0$. On the other hand, by the definition of $H$, there holds
    \begin{align*}
        H(\pi_{k+1}(\cdot|s)) &= \underbrace{\big(\pi_{k+1}(a_3|s)\big)^4}_{ > 0} \cdot \Bigg\{ \underbrace{\frac{-\pi_{k+1}(a_2|s)}{\sqrt{1-(\pi_{k+1}(a_2|s))^2}}}_{\leq -0.8 / \sqrt{0.36} =-4/3} - \frac{1}{4} \underbrace{\sin \big( \log (1-\log \pi_{k+1}(a_3|s)) \big)}_{\geq -1} \Bigg\}  < 0.
    \end{align*}
    This contradicts $H(\pi_{k+1}(\cdot|s)) = H(\pi_k(\cdot|s)) + \lambda_1^{(k)}/2 > H(\pi_k(\cdot|s)) = 0$ given by Eq~\eqref{eq:potential-dynamic}, thus $\lambda_1^{(k)} = 0$. Following the same argument, one can similarly show that $\lambda_2^{(k)} = 0$ via contradiction. Therefore, we conclude that $\lambda_1^{(k)} = \lambda_2 ^{(k)} = 0$ and hence $H(\pi_{k+1}(\cdot|s)) = 0$, which proves Eq~\eqref{eq:potential} by induction.

    Now we prove Eq~\eqref{eq:policy-expression}. Note that $\pi_k(a_3|s) > 0$. By $H(\pi_k(\cdot|s)) = 0$, we obtain
    \begin{alignat*}{2}
        &&\frac{\pi_k(a_1|s) - \pi_k(a_2|s)}{\sqrt{1-(\pi_k(a_1|s) - \pi_k(a_2|s))^2}} &= \frac{1}{4} \sin \ls( \log(1-\log \pi_k(a_3|s)) \rs) \\
        \Longrightarrow \quad \quad && \pi_k(a_1|s) - \pi_k(a_2|s) &= \frac{\sin(\log(1 - \log \pi_k(a_3\vert{}s)))}{\sqrt{16 + \sin^2(\log(1 - \log \pi_k(a_3\vert{}s)))}} =: u_k
    \end{alignat*}
    Together with $\pi_k(a_1|s) + \pi_k(a_2|s) = 1-\pi_k(a_3|s)$ we get the desired expression in Eq~\eqref{eq:policy-expression}.
\end{proof}

\begin{lemma}  \label{lem:phase}
    Under the initialization in Eq~\eqref{eq:init}, there exists a finite time $T_1 > 0$ such that
    \begin{align}  \label{eq:suboptimal-decay}
        \forall\, k \geq T_1: \quad \eta / 2 \leq \log \pi_{k}(a_3|s) - \log \pi_{k+1}(a_3|s) \leq 3\eta / 2.
    \end{align}
    Let $\theta_k := \log (1 - \log \pi_k(a_3|s))$. Eq~\eqref{eq:suboptimal-decay} implies that
    \begin{align}  \label{eq:phase-limits}
        \lim_{k \to \infty} \theta_k = +\infty, \quad \lim_{k\to\infty} ( \theta_{k+1} - \theta_k )= 0.
    \end{align}
\end{lemma}

\begin{proof}
    By Eq~\eqref{eq:potential-equations} there holds \[\nu^{(k)} = -\eta - \frac{(\lambda_1^{(k)} + \lambda_2^{(k)})}{2} = -\eta,\] 
    as $\lambda_1^{(k)} = \lambda_2^{(k)} = 0$. Noticing that $\pi_{k+1}(a_3|s) > 0$, we have $\lambda_3^{(k)} = 0$ and thus by Eq~\eqref{eq:KKT}
    \begin{align}  \label{eq:suboptimal-policy-dynamic}
        \forall\, k \geq 0: \quad \frac{\partial h}{\partial p_3} (\pi_{k+1}(a_3|s)) = \frac{\partial h}{\partial p_3} (\pi_k(a_3|s)) - \eta.
    \end{align}
    Define the remainder term
    \begin{align*}
        R_k &:= 200 \pi_k(a_3|s) + \frac{\pi_k(a_3|s)}{\sqrt{4- \pi_k(a_3|s)^2}} - 4 \big(\pi_k(a_3|s)\big)^3 \sqrt{1-\big( \pi_k(a_1|s) - \pi_k(a_2|s) \big)^2} \\
        &\;\;\;\;- \big( \pi_k(a_1|s) - \pi_k(a_2|s) \big) \big( \pi_k(a_3|s) \big)^3 \ls[ \sin (\log (1 - \log \pi_k(a_3|s))) - \frac{\cos (\log (1- \log \pi_k(a_3|s)))}{4 (1 - \log \pi_k(a_3|s))} \rs].
    \end{align*}
    Plugging Eq~\eqref{eq:gradients} into Eq~\eqref{eq:suboptimal-policy-dynamic} yields that
    \begin{align*}
        \log \pi_{k+1}(a_3|s) + R_{k+1} = \log\pi_k(a_3|s) + R_{k} - \eta.
    \end{align*}
    Notice that $\pi_k(a_3|s) \to 0$ by the global convergence of PMD (Proposition~\ref{pro:Xiao_Lan}), and $R_k \to 0$ as $\pi_k(a_3|s) \to 0$. Therefore, for any fixed step size $\eta > 0$, there exists a finite time $T_1 > 0$ such that $-\eta /4 \leq R_k \leq \eta /4$ for all $k \geq T_1$. Thus
    \begin{align*}
        \forall\, k \geq T_1 : \quad \eta / 2 \leq \log\pi_k(a_3|s) - \log \pi_{k+1}(a_3|s) \leq 3\eta / 2.
    \end{align*} 
    As $\pi_k(a_3|s) \to 0$, there holds $\theta_k = \log (1 - \log \pi_k(a_3|s)) \to +\infty$. Furthermore,
    \begin{align*}
        \theta_{k+1} - \theta_k = \log \ls( \frac{1-\log \pi_{k+1}(a_3|s)}{1 - \log \pi_k(a_3|s)} \rs) \in \ls[ \underbrace{\log \ls( 1 + \frac{\eta/2}{1-\log \pi_k(a_3|s)} \rs)}_{\to 0}, \; \underbrace{\log \ls( 1 + \frac{3\eta/2}{1-\log \pi_k(a_3|s)} \rs)}_{\to 0} \rs],
    \end{align*}
    implying that $\theta_{k+1} - \theta_k \to 0$.
\end{proof}

\begin{lemma}\label{lem:periodic_density}
Let $f:\mathbb R\to\mathbb R$ be a continuous periodic function with period $T>0$, and let $\{\theta_k\}_{k\geq1}$ satisfy
\[
\theta_k\to+\infty,\qquad |\theta_{k+1}-\theta_k|\to0.
\]
Then for any fixed $a \in [0, T)$, there exists a subsequence $\{\theta_{k_m}\}_{m \geq 0}$ such that
\[
\lim_{m \to \infty} f(\theta_{k_m}) = f(a).
\]
Equivalently, $\{f(\theta_k)\}_{k\geq1}$ is dense in the range of $f$.
\end{lemma}

\begin{proof}
Fix any $a\in[0,T)$. For every sufficiently large integer $m$, define
\[
k_m:=\min\{k\geq1: \;\;\theta_k\geq mT+a\}.
\]
This index is well defined because $\theta_k\to+\infty$. Moreover, $k_m\to\infty$, and, by the minimality of $k_m$,
\[
\theta_{k_m-1}<mT+a\leq\theta_{k_m}.
\]
It follows that
\[
0\leq\theta_{k_m}-(mT+a)\leq\theta_{k_m}-\theta_{k_m-1}\leq|\theta_{k_m}-\theta_{k_m-1}|\to0.
\]
Hence $\theta_{k_m}-mT\to a$. By the periodicity and continuity of $f$,
\[
f(\theta_{k_m})=f(\theta_{k_m}-mT)\to f(a).
\]
\end{proof}

\begin{theorem}
    Consider the bandit problem described in Eq~\eqref{eq:bandit} and the mirror map $h$ defined in Eq~\eqref{eq:counter-h}. Setting the initialization in Eq~\eqref{eq:init}, the policy sequence generated by PMD, denoted $\{ \pi_k \}_{k \geq 0}$, does not converge to any fixed optimal policy. Specifically, for any $y \in [-1/\sqrt{17}, \; 1/\sqrt{17}]$, there exists a subsequence $\{ \pi_{k_m} \}_{m \geq 0}$ such that
    \begin{align*}
        \lim_{m \to \infty} \pi_{k_m}(\cdot|s) = \left [ \begin{array}{c}
             (1+y) / 2  \\
             (1-y) / 2  \\
             0
        \end{array}
        \right ].
    \end{align*}
\end{theorem}
\begin{proof}
    By Lemma~\ref{lem:potential},
    \begin{align*}
        \pi_k(a_1 |s) = \frac{(1-\pi_k(a_3|s)) + u_k}{2}, \quad \pi_k(a_2|s) = \frac{(1 - \pi_k(a_3|s)) - u_k}{2},
    \end{align*}
    where
    \begin{align*}
        u_k = \frac{\sin (\theta_k)}{\sqrt{16 + \sin^2(\theta_k)}}, \quad \theta_k = \log (1 - \log \pi_k(a_3|s)).
    \end{align*}
    Let $f(\theta) = {\sin (\theta)}\big /{\sqrt{16 + \sin^2(\theta)}}$. By Lemmas~\ref{lem:phase} and~\ref{lem:periodic_density}, for any fixed $y \in [-1/\sqrt{17}, \; 1/\sqrt{17}]$, there exist an $a \in [0, 2\pi)$ such that $f(a) = y$ and a subsequence $\{ \theta_{k_m} \}_{m \geq 0}$ satisfying $f(\theta_{k_m}) \to y$. Together with $\pi_k(a_3|s) \to 0$, we obtain
    \begin{align*}
        \lim_{m \to \infty} \pi_{k_m}(a_1|s) = \lim_{m \to \infty} \frac{1 - \pi_{k_m}(a_3|s) + f(\theta_{k_m})}{2} = \frac{1+y}{2}, \quad \lim_{m\to\infty} \pi_{k_m}(a_2|s) = \frac{1-y}{2},
    \end{align*}
    which completes the proof.
\end{proof}

\section{Conclusion}
This note constructs a Legendre but non-decomposable mirror map and a three-armed bandit problem for which PMD fails to converge in the policy domain under any positive constant step size and a suitable initialization. Although the value sequence converges to the optimal value, the policy  sequence has a nontrivial continuum of optimal accumulation points and therefore does not converge in the policy domain. This counterexample shows that the Legendre conditions alone are insufficient to guarantee policy convergence of PMD, and that additional structural conditions, such as decomposability, are needed for a general policy convergence guarantee.

\section*{Declaration of Generative AI Assistance}

The authors have reviewed and verified all mathematical arguments and results, reorganized the exposition, and provided additional explanations. They take full responsibility for the content of this note.

\bibliographystyle{plainnat}
\bibliography{refs}

\appendix
\section*{Appendix: Computing the Hessian $\nabla^2 h$}
  \label{sec:app:Hessian}
Let $h_{ij} := \frac{\partial^2 h}{\partial p_i \partial p_j} (p)$. Direct computation yields
\begin{align*}
    h_{11}=h_{22}&=1+\frac{p_3^4}{[1-(p_1-p_2)^2]^{3/2}}, \qquad h_{12}=1-\frac{p_3^4}{[1-(p_1-p_2)^2]^{3/2}},\\
    h_{13}&=1+\frac{4p_3^3(p_1-p_2)}{\sqrt{1-(p_1-p_2)^2}}-p_3^3\sin\!\bigl(\log(1-\log p_3)\bigr)+\frac{p_3^3\cos\!\bigl(\log(1-\log p_3)\bigr)}{4(1-\log p_3)},\\
    h_{23}&=1-\frac{4p_3^3(p_1-p_2)}{\sqrt{1-(p_1-p_2)^2}}+p_3^3\sin\!\bigl(\log(1-\log p_3)\bigr)-\frac{p_3^3\cos\!\bigl(\log(1-\log p_3)\bigr)}{4(1-\log p_3)},\\
    h_{33}&=1+\frac{1}{p_3}+200+\frac{4}{(4-p_3^2)^{3/2}}-12p_3^2\sqrt{1-(p_1-p_2)^2}\\
    &\quad -(p_1-p_2)p_3^2\Bigg[3\sin\!\bigl(\log(1-\log p_3)\bigr)-\frac{7\cos\!\bigl(\log(1-\log p_3)\bigr)}{4(1-\log p_3)} \\
    &\quad \quad-\frac{\sin\!\bigl(\log(1-\log p_3)\bigr)+\cos\!\bigl(\log(1-\log p_3)\bigr)}{4(1-\log p_3)^2}\Bigg],
\end{align*}
where $p \in \interior\, \dom \, h = \big\{ p\in\mathbb{R}^3 \, | \, 0 < p_3 < 2, \;\; |p_1 - p_2| < 1 \big\}$. We show the positive definiteness of $\nabla^2 h$ by verifying that all leading principal minors are positive. For the first two leading principal minors,
\[
\Delta_1=h_{11}=1+\frac{p_3^4}{[1-(p_1-p_2)^2]^{3/2}}>0,
\]
and
\[
\Delta_2=h_{11}h_{22}-h_{12}^2=\frac{4p_3^4}{[1-(p_1-p_2)^2]^{3/2}}>0.
\]
Since $h_{23}=2-h_{13}$, direct calculation provides that
\[
\Delta_3=\det\!\bigl(\nabla^2h(p)\bigr)=\frac{4p_3^4}{[1-(p_1-p_2)^2]^{3/2}}\left\{h_{33}-1-\frac{(h_{13}-1)^2}{p_3^4/[1-(p_1-p_2)^2]^{3/2}}\right\}.
\]
To bound the expression in braces, note that $1-\log p_3>1-\log 2$ and hence
\[
\left|\sin\!\bigl(\log(1-\log p_3)\bigr)-\frac{\cos\!\bigl(\log(1-\log p_3)\bigr)}{4(1-\log p_3)}\right|
\leq 1+\frac{1}{4(1-\log 2)}<1.815,
\]
and
\begin{align*}
\Bigg|3\sin\!\bigl(\log(1-\log p_3)\bigr)-\frac{7\cos\!\bigl(\log(1-\log p_3)\bigr)}{4(1-\log p_3)}&-\frac{\sin\!\bigl(\log(1-\log p_3)\bigr)+\cos\!\bigl(\log(1-\log p_3)\bigr)}{4(1-\log p_3)^2}\Bigg|
\\
&\leq 3+\frac{7}{4(1-\log 2)}+\frac{\sqrt{2}}{4(1-\log 2)^2}<12.458.
\end{align*}
It follows that
\[
h_{33}-1>200-48-4(12.458)=102.168.
\]
Moreover,
\begin{align*}
\frac{(h_{13}-1)^2}{p_3^4/[1-(p_1-p_2)^2]^{3/2}}
&=p_3^2\sqrt{1-(p_1-p_2)^2}\Bigg[4(p_1-p_2)-\sqrt{1-(p_1-p_2)^2}\\
&\qquad \times\left(\sin\!\bigl(\log(1-\log p_3)\bigr)-\frac{\cos\!\bigl(\log(1-\log p_3)\bigr)}{4(1-\log p_3)}\right)\Bigg]^2\\
&\leq 4\sqrt{1-(p_1-p_2)^2}\left[4|p_1-p_2|+\left(1+\frac{1}{4(1-\log 2)}\right)\sqrt{1-(p_1-p_2)^2}\right]^2\\
&\leq 4\left[32(p_1-p_2)^2\sqrt{1-(p_1-p_2)^2}+2\left(1+\frac{1}{4(1-\log 2)}\right)^2\bigl(1-(p_1-p_2)^2\bigr)^{3/2}\right]\\
&\leq 4\left[32\left(\frac{2}{3\sqrt{3}}\right)+2\left(1+\frac{1}{4(1-\log 2)}\right)^2\right]<75.613,
\end{align*}
where we used $\max_{|x|\leq1}x^2\sqrt{1-x^2}=2/(3\sqrt{3})$. Therefore,
\[
h_{33}-1-\frac{(h_{13}-1)^2}{p_3^4/[1-(p_1-p_2)^2]^{3/2}}>102.168-75.613=26.555>0,
\]
which implies
\[
\Delta_3>\frac{4p_3^4}{[1-(p_1-p_2)^2]^{3/2}}\times 26.555>0.
\]
As all leading principal minors of $\nabla^2 h$ are positive, we have $\nabla^2 h(p) \succ 0$ for all $p \in \interior\, \dom \, h$.

\end{document}